\documentclass[dvipdfm]{amsart}

\usepackage{amsmath, amsthm}
\usepackage{amsmath, amsfonts}
\usepackage{amsmath, amssymb}
\usepackage{amsmath}
\usepackage[all]{xy}

\newtheorem{thm}{Theorem}[section]
\newtheorem{cor}[thm]{Corollary}
\newtheorem{lem}[thm]{Lemma}

\newtheorem{rmk}[thm]{Remark}

\numberwithin{equation}{section}

\newcommand{\spec}{\mbox{Spec}}

\newcommand{\EE}{\mathbb{E}}
\newcommand{\XX}{\mathfrak{X}}
\newcommand{\MM}{\mathcal{M}}
\newcommand{\GG}{\mathbb{G}}

\newcommand{\Hom}{\mathcal{H}om}

\newcommand{\RHom}{R\mathcal{H}om}

\newcommand{\Lotimes}{\stackrel{L}{\otimes}}

\newcommand{\bl}{\textbf}

\newcommand{\Ext}{\mbox{Ext}}

 \makeatletter
 
 \newcommand{\Rmnum}[1]{\expandafter\@slowromancap\romannumeral #1@}
 \makeatother
\begin{document}

  %封面内容
  \title{On the deformation theory of pair $(X, E)$}
  \author {Si\ \ \  Li}
  \address{Mathematics Department, Harvard University}
  \email{sili@math.harvard.edu}
\date{}
  % 封面
  \maketitle

%%%%%%%%%%%%%%%%%%%%%%%%%%%%%%
%% 前言部分
%%%%%%%%%%%%%%%%%%%%%%%%%%%%%%

\begin{abstract}Huybrechts and Thomas recently constructed
relative obstruction theory of objects of the derived category of
coherent sheaves over smooth projective family. In this paper, we
use this construction to obtain the absolute deformation-obstruction
theory of the pair $(X, E)$, with X smooth projective scheme and $E$
perfect complex, and show that the obstruction theories for $E, (X,
E),$ and $X$ fit into exact triangle as derived objects on the
moduli space.
\end{abstract}

  % 目录
%  \tableofcontents
  % 表格目录
%  \listoftables
  % 插图目录
%  \listoffigures
\section{Introduction}

\indent\indent The deformation theory of objects of the derived
category of coherent sheaves on smooth projective variety has been
studies in \cite{Lieblich, Lowen} and developed recently in
\cite{Hurbrechts-Thomas}. The latter uses Illusie's cotangent
complex and Atiyah class \cite{Ill} to show that the obstruction
class is the product of Atiyah and Kodaira-Spencer classes, and
describe the relative obstruction theory (in the sense of
\cite{intrinsic}) for moduli space of perfect simple complexes on
smooth projective families of threefold, which is used to obtain
virtual cycle generalizing the virtual counting in
\cite{Stable-Pair} and \cite{Thomas}.\\

 In this paper, we show how the setting of relative obstruction
theory in \cite{Hurbrechts-Thomas} can be used to obtain the
absolute obstruction theory of the pair $(X, E)$, with X smooth
projective scheme and $E$ perfect complex of coherent sheaves.
Specifically, given a perfect complex $E$ on X, Illusie's Atiyah
class gives an element
$$
    A(E)\in \Ext^1_X(E, E\otimes L^\bullet_X)
$$
where $L^\bullet_X$ is Illusie's cotangent complex, which is
quasi-isomorphic to the cotangent bundle $\Omega_X$ in our case when
$X$ is smooth. View it as a map in the derived category
$$
    A(E): \RHom(E, E)[-1] \to \Omega_X
$$
and let $G$ be the mapping cone. Then the tangent space of deforming
the pair is
$$
    \Ext^1_X(G, \mathcal O_X)
$$
and the obstruction space lies in
$$
    \Ext^2_X(G, \mathcal O_X)
$$
The exact triangle
$$
    \RHom(E, E)[-1] \to \Omega_X\to G\to \RHom(E, E)
$$
naturally puts the tangent-obstruction spaces of deforming the
complex $E$ fixing $X$, deforming the pair $(X, E)$ and deforming
the scheme $X$ into a long exact sequence.\\

 As an application, we specialize to the case that $E$ is a vector
bundle on X. The complex $G$ can be explicitly described in this
case, and we recover the fact that the tangent and obstructio space
for deforming the pair $(X, E)$ is obtained via the first and second
cohomology of the sheaf of differential operators of order $\leq 1$ with diagonal symbol (\ref{DE}).\\

\bl{Notation.} $k$ is fixed to be algebraic closed field of
characteristic zero. We use standard notations for derived functors,
for instance, $L\pi^*$ is the derived pull-back by $\pi$, $R\pi_*$
is the derived push-forward by $\pi$, $(\cdot)^v$ is the derived
dual,  and $\RHom$ is the derived
$\mathcal Hom$.\\

{\bf Acknowledgements.} The author thanks R.P.\ Thomas, D.\
Huybrechts for valuable
 comments, and his advisor S.T.\ Yau for constant support.

\section{Moduli of Pair and Relative obstruction theory}
In this section, we consider
\begin{eqnarray*}
        &\mathfrak{X}&\\ &\downarrow& \\ &S&
\end{eqnarray*}
a flat family of smooth projective varieties of dimension n which is
universal at every point of $S$. Denote by $i_s:
\XX_s\hookrightarrow \XX$ the fiber of $\XX$ over a closed point
$s\in S$ and similar notation for other families. Let $\mathcal M/S$
be a \emph{relative fine moduli space of perfect complexes} over
$\mathfrak X/S$ (see \cite{Hurbrechts-Thomas} for more detail).
There is a perfect complex (the universal complex for the moduli
space)
\begin{eqnarray}
    \EE\in D^b(\MM\times_S \XX)
\end{eqnarray}
such that $\MM/S$ represents the functor which associates any scheme
$T$ over $S$ the set of equivalence classes of perfect complexes $E$
over $T\times_S \XX$
\begin{eqnarray*}
            &E& \\ &\downarrow& \\ &T\times_S \XX&
\end{eqnarray*}
whose restriction to any fiber to $T\times_S \XX/T$ is isomorphic to
the restriction of $\EE$ to some fiber of $\MM\times_S \XX/\MM$.\\
\indent Consider the cartesian diagram
\begin{eqnarray*}
    \xymatrix{
    \mathcal M\times_S \mathfrak{X}\ar[d]_{\pi_{\mathcal M}}
    \ar[r]^{p_{\mathfrak X}}
    &  \mathfrak X\ar[d]^{\pi_S}\\
    \mathcal M\ar[r]^{p_S} & S
    }
\end{eqnarray*}
where we denote by $\pi_{\MM}, \pi_S, p_{\XX}, p_S$ the
corresponding morphisms as in the diagram. Let $m\in \MM$ be a
closed point of $\MM$, $s=p_S(m)$, $\EE_m=i_m^* \EE$ the restriction
of $\EE$ over $m$, then we get a pair
\begin{eqnarray*}
        (\XX_s, \EE_m)
\end{eqnarray*}
Since we assume that $S$ is universal at $s$, the moduli space $\MM$
at $m$ actually parameterizes the local deformation space of the
pair
$(\XX_s, \EE_m)$.\\
\indent A relative obstruction theory for $\MM/S$ is constructed via
Atiyah class in \cite{Hurbrechts-Thomas}. We review their
construction which will be generalized in the next section to the
absolute case. Let
$$
    A(\EE)\in \Ext^1_{\MM\times_S \XX} (\EE, \EE\Lotimes L^\bullet_{\MM\times_S \XX})
$$
be Illusie's Atiyah class. Here $L^\bullet_{\MM\times_S \XX}$ is the
cotangent complex. We denote by
\begin{eqnarray*}
 A_{\pi_{\MM}}(\EE)\in \Ext^1_{\MM\times_S \XX}(\EE, \EE\Lotimes L^\bullet_{\MM\times_S
 \XX/\MM})\\
  A_{p_{\XX}}(\EE)\in \Ext^1_{\MM\times_S \XX} (\EE, \EE\Lotimes L^\bullet_{\MM\times_S \XX/\XX})
\end{eqnarray*}
the image of $A(\EE)$ via the two maps of cotangent complexes
\begin{eqnarray*}
    L^\bullet_{\MM\times_S \XX}\to L^\bullet_{\MM\times_S\XX/\MM}\\
    L^\bullet_{\MM\times_S \XX}\to L^\bullet_{\MM\times_S\XX/\XX}
\end{eqnarray*}
Since the map $\pi_S$ is flat, we actually have canonical
isomorphisms
$$
L^\bullet_{\MM\times_S\XX/\MM}\simeq Lp_{\XX}^*(L^\bullet_{\XX/S}),
\ \ \ \ \ L^\bullet_{\MM\times_S\XX/\XX}\simeq
L\pi_{\MM}^*(L^\bullet_{\MM/S})
$$
 The class $A_{p_{\XX}}(\EE)$ gives a map in the derived
category
$$
    \RHom(\EE, \EE)[-1]\to L\pi_{\MM}^* L^\bullet_{\MM/S}
$$
which by Verdier duality along the projective morphism $\pi_{\MM}$
gives a map
\begin{eqnarray}\label{relative theory}
    R\pi_{\MM*}(\RHom(\EE, \EE)\Lotimes \omega_{\pi_{\MM}})[n-1]\to
    L^\bullet_{\MM/S}
\end{eqnarray}
where $\omega_{\pi_{\MM}}$ is the relative dualizing sheaf along
$\pi_{\MM}$.
\begin{thm}[\cite{Hurbrechts-Thomas}]
    The map (\ref{relative theory}) is a relative obstruction theory
    for $\MM/S$.
\end{thm}
This means the map (\ref{relative theory}) has the property that
$h^{-1}$ is epimorphism, $h^0$ is isomorphism \cite{intrinsic}.

\begin{rmk}
Note that only part of the full Atiyah class $A(\EE)$, i.e.
$A_{p_{\XX}}(\EE)$ is used to obtain the relative obstruction
theory. The other part $A_{\pi_{\MM}}(\EE)$ will also be used to
obtain the absolute obstruction theory as we will show in the next
section.
\end{rmk}

\section{Atiyah Class and Obstruction theory of the Pair}
We will keep the same notation in this section as above. Since
$\XX/S$ is a smooth family, the relative cotangent complex
$L^\bullet_{\XX/S}$ is in fact isomorphic to the one-term locally
free sheaf of relative differentials
$$
L^\bullet_{\XX/S}\simeq \Omega_{\XX/S}
$$
The Atiyah class $A_{\pi_{\MM}}(\EE)$ can be written explicitly as
an exact sequence of complexes
\begin{eqnarray}\label{ex seq of Atiyah}
A_{\pi_{\MM}}(\EE):  0\to L^\bullet_{\MM\times_S\XX/\MM}\otimes \EE
\to \EE_{\pi_{\MM}}\to \EE\to 0
\end{eqnarray}
where $\EE_{\pi_{\MM}}$ is isomorphic to $\EE\oplus
\Omega_{\MM\times_S\XX/\MM}\otimes \EE\simeq \EE\oplus
L^\bullet_{\MM\times_S\XX/\MM}\otimes \EE$ as $k$-linear spaces, but
with $\mathcal O_{\MM\times_S \XX}$-module structure given by
\begin{eqnarray}\label{structure-atiyah}
    a\cdot (e_1\oplus e_2\otimes db)=(a\,e_1\oplus a\,e_2\otimes db+e_1\otimes da)
\end{eqnarray}
for $a, b\in \mathcal O_{\MM\times_S \XX}, e_1, e_2\in \EE$. Use the
canonical isomorphism
$$
    \Ext^1_{\MM\times_S\XX}(\EE, \EE\Lotimes L^\bullet_{\MM\times_S\XX})\simeq
    \Ext^1_{\MM\times_S\XX}(\RHom(\EE, \EE), L^\bullet_{\MM\times_S\XX})
$$
we can write $A_{\pi_{\MM}}(\EE)$ as a map
$$
   A_{\pi_{\MM}}(\EE):\ \ \  \RHom(\EE, \EE)[-1] \to L^\bullet_{\MM\times_S\XX/\MM}
$$
%Explicitly, apply $\RHom(\EE, \cdot)$ to the exact sequence (\ref{ex
%seq of Atiyah}), we get
%\begin{eqnarray*}
%0\to L^\bullet_{\MM\times_S\XX/\MM}\otimes \RHom(\EE, \EE) \to
%\RHom(\EE, \EE_{\pi_{\MM}})\to \RHom(\EE, \EE)\to 0
%\end{eqnarray*}
%whose push-forward under the trace map gives the required extension
%class in $\Ext^1(\RHom(\EE, \EE), L^\bullet_{\MM\times_S\XX/MM})$
%\begin{eqnarray*}
%\xymatrix{ 0\ar[r]& L^\bullet_{\MM\times_S\XX/\MM}\otimes \RHom(\EE,
%\EE) \ar[r]\ar[d]_{1\otimes tr} & \RHom(\EE,
%\EE_{\pi_{\MM}})\ar[r]\ar[d]& \RHom(\EE, \EE)\ar@{=}[d]\ar[r] & 0\\
%0\ar[r] & L^\bullet_{\MM\times_S\XX/\MM} \ar[r] & \star\ar[r] &
%\RHom(\EE, \EE)\ar[r] & 0 }
%\end{eqnarray*}

 We define the complex $\GG$ to be the mapping cone of the above
map. We get exact triangle
\begin{eqnarray}\label{exact sequence}
\RHom(\EE, \EE)[-1] \to L^\bullet_{\MM\times_S\XX/\MM}\to \GG\to
\RHom(\EE, \EE)
\end{eqnarray}

%\begin{lem}
%    The dual $\GG^v$ is isomorphic as object in the derived category
%    to the following complex ???????????????????????????
%\end{lem}
Note that we have commutative diagram of cotangent complexes
\begin{eqnarray*}
    \xymatrix{
        L^\bullet_{\MM\times_S\XX}\ar[r]\ar[d]&
        L^\bullet_{\MM\times_S\XX/\MM}\ar[d]\\
        L^\bullet_{\MM\times_S\XX/\XX}\ar[r] & L\pi_{\MM}^* Lp_S^*
        L^\bullet_S[1] \simeq Lp_{\XX}^*L\pi_S^* L^\bullet_S[1]
    }
\end{eqnarray*}
Combined with the Atiyah class $\RHom(\EE, \EE)[-1]\to
L^\bullet_{\MM\times_S \XX }$, we get commutative diagram\\
\begin{eqnarray*}
    \xymatrix{
       \RHom(\EE, \EE)\ar[r]\ar[d]&
        L^\bullet_{\MM\times_S\XX/\MM}[1]\ar[d]\\
        L^\bullet_{\MM\times_S\XX/\XX}[1]\ar[r] & L\pi_{\MM}^*Lp_S^*
        L^\bullet_S[2] \simeq Lp_{\XX}^*L\pi_S^* L^\bullet_S[2]
    }
\end{eqnarray*}
Note that we also have exact triangle
\begin{eqnarray*}
    Lp_S^* L^\bullet_{S}\to L^\bullet_{\MM}\to L^\bullet_{\MM/S}\to
    Lp_S^* L^\bullet_S[1]
\end{eqnarray*}
Pull it back to $\MM\times_S \XX$ via $\pi_{\MM}$ we get exact
triangle
\begin{eqnarray*}
  L\pi_{\MM}^*Lp_{S}^* L^\bullet_{S}\to  L\pi_{\MM}^* L^\bullet_{\MM}\to L\pi_{\MM}^*
    L^\bullet_{\MM/S}\simeq L^\bullet_{\MM\times_S \XX/\XX}\to
   L \pi_{\MM}^*Lp_{S}^* L^\bullet_{S}[1]
\end{eqnarray*}
hence the above commutative diagram can be fit into maps of exact
triangles
\begin{eqnarray*}
\xymatrix{ L^\bullet_{\MM\times_S\XX/\MM}\ar[r]\ar[d] &
\GG\ar[r]\ar[d]& \RHom(\EE, \EE)\ar[d]^{A_{p_{\XX}}(\EE)}\ar[r]^{A_{\pi_{\MM}}(\EE)} & L^\bullet_{\MM\times_S\XX/\MM}[1] \ar[d] \\
  L\pi_{\MM}^*Lp_{S}^* L^\bullet_{S}[1]\ar[r]& L \pi_{\MM}^* L^\bullet_{\MM}[1]\ar[r] & L\pi_{\MM}^*
    L^\bullet_{\MM/S}[1]\simeq L^\bullet_{\MM\times_S \XX/\XX}[1]\ar[r] &
    L\pi_{\MM}^*Lp_{S}^* L^\bullet_{S}[2]
 }
\end{eqnarray*}

By Verdier duality along the projective morphism $\pi_{\MM}$, we get
\begin{eqnarray*}
\xymatrix{ R\pi_{\MM*}(L^\bullet_{\MM\times_S\XX/\MM}\Lotimes
\omega_{\pi_{\MM}})[n-1]\ar[r]\ar[d] & R\pi_{\MM*}(\GG\Lotimes
\omega_{\pi_{\MM}})[n-1]\ar[r]\ar[d]& R\pi_{\MM*}(\RHom(\EE,
\EE\Lotimes
\omega_{\pi_{\MM}})[n-1]\ar[d]\\
Lp_{S}^* L^\bullet_{S}\ar[r]& L^\bullet_{\MM}\ar[r] &
    L^\bullet_{\MM/S}
}
\end{eqnarray*}
where $\omega_{\pi_{\MM}}$ is the relative dualizing sheaf along
$\pi_{\MM}$.
\begin{thm}\label{Main Thm}
The map
\begin{eqnarray}
    R\pi_{\MM*}(\GG\Lotimes \omega_{\pi_{\MM}})[n-1]\to L^\bullet_{\MM}
\end{eqnarray}
gives an obstruction theory for $\MM$.
\end{thm}
\begin{proof}
    Let's first consider the map
    \begin{eqnarray}\label{derived KS}
        R\pi_{\MM*}(L^\bullet_{\MM\times_S\XX/\MM}\Lotimes \omega_{\pi_{\MM}})[n-1]\to Lp_S^*
        L^\bullet_{S}
    \end{eqnarray}
Let $\omega_{\pi_S}$ be the relative dualizing sheaf of $\XX/S$,
then
$$
    \omega_{\pi_{\MM}}=Lp_{\XX}^* \omega_{\pi_S}
$$
we have
\begin{eqnarray*}
    L^\bullet_{\MM\times_S\XX/\MM}\Lotimes
    \omega_{\pi_{\MM}}=Lp_{\XX}^*(L^\bullet_{\XX/S}\Lotimes \omega_{\pi_S} )
\end{eqnarray*}
since $\pi_S$ is flat, we have base change property
$$
    R\pi_{\MM*}(L^\bullet_{\MM\times_S\XX/\MM}\Lotimes
    \omega_{\pi_{\MM}})[n-1] \simeq  Lp_S^*(R\pi_{S*}( L^\bullet_{\XX/S}\Lotimes \omega_{\pi_S}
    ))[n-1]
$$
and it's easy to see that the map (\ref{derived KS}) is the
pull-back by $p_S$ of the Kodaira-Spencer map \cite{intrinsic}
\begin{eqnarray}\label{KS on S}
R\pi_{S*}( L^\bullet_{\XX/S}\Lotimes \omega_{\pi_S}
    )[n-1]\to L^\bullet_S
\end{eqnarray}
By Prop 6.2 in \cite{intrinsic}, the Kodaira-Spencer map (\ref{KS on
S}) gives an obstruction theory on $S$, i.e., $h^{-1}$ is
epimorphism and $h^0$ is isomorphism. It follows that the map
(\ref{derived KS}) is also epimorphism for $h^{-1}$ and isomorphism
for $h^0$.\\
\indent By Theorem 4.1 in \cite{Hurbrechts-Thomas}, the map
$$
    R\pi_{\MM*}(\RHom(\EE, \EE)\Lotimes \omega_{\pi_{\MM}})[n-1]\to
    L^\bullet_{\MM/S}
$$
gives a relative obstruction theory for $\MM/S$, hence epimorphism
for $h^{-1}$ and isomorphism for $h^0$. Using the long exact
sequence of cohomology associated to the exact triangle, and by
simple diagram chasing, we see that the map
\begin{eqnarray*}
    R\pi_{\MM*}(\GG\Lotimes \omega_{\pi_{\MM}})[n-1]\to L^\bullet_{\MM}
\end{eqnarray*}
is also epimorphism for $h^{-1}$ and isomorphism for $h^0$, hence
giving an obstruction theory for $\MM$.
\end{proof}

\begin{cor}\label{main cor}
    Let $X$ be smooth projective variety, $E\in D^b(X)$ be a perfect
    complex. Let $G^\bullet$ be the mapping cone of the Atiyah class
    $$
            \RHom(E, E)[-1]\to \Omega_X
    $$
    Then the deformation functor $Def_{(X, E)}$ of the pair $(X, E)$ has
    tangent space
    $$
       \emph{ \Ext}^1_X(G^\bullet, \mathcal O_X)
    $$
    and obstruction space can be chosen to be
    $$
       \emph{ \Ext}^2_X(G^\bullet, \mathcal O_X)
    $$
    and for any small extension $0\to k\to A^\prime \to A \to 0$,
    where $A^\prime, A$ are Artin local rings, the restriction map
    $$
        Def_{(X, E)}(A^\prime)\to Def_{(X, E)}(A)
    $$
    is a torsor under
    $$
        \emph{\Ext}^0_X(G^\bullet, \mathcal O_X)
    $$
\end{cor}
\begin{proof}
        It follows from Theorem 4.5 in
        \cite{intrinsic}, Theorem \ref{Main Thm} and Serre Duality.
\end{proof}
\section{Application: Deformation theory of Vector Bundle on Smooth Projective Variety}
In this section, we specialize the above discussion to the case of
pair $(X, E)$, where X is projective smooth variety, and E is a
vector bundle on X. Let
$$
    0\to \Omega_X\otimes E\to E_A\to E\to 0
$$
be the Atiyah class. Apply $\Hom(\cdot, E)$, we get
\begin{eqnarray}\label{Hom E of Atiyha}
    0\to \Hom(E, E)\stackrel{i}{\to} \Hom(E_A, E)\stackrel{j}{\to} \Hom(E\otimes \Omega_X, E)\to 0
\end{eqnarray}
We denote the following canonical diagonal map by $k$
$$
    k: \Hom(\Omega_X, \mathcal O_X)\to \Hom(E\otimes \Omega_X, E)
$$
The dual of Atiyah class as an element in $\Ext^1_X(\Hom(\Omega_X,
\mathcal O_X), \Hom(E, E))$ is obtained via the pull-back of the
exact sequence (\ref{Hom E of Atiyha}) by the diagonal map k
\begin{eqnarray*}
    \xymatrix{
    0\ar[r]& \Hom(E, E)\ar[r]^i & \Hom(E_A, E)\ar[r]^j & \Hom(E\otimes \Omega_X,
    E)\ar[r] &
    0\\
    0\ar[r] & \Hom(E, E)\ar[r]\ar[u]^{=}& D(E)\ar[r]\ar[u] & \Hom(\Omega_X, \mathcal
    O_X)\ar[r]\ar[u]^{k}& 0
    }
\end{eqnarray*}
where
\begin{eqnarray}\label{DE}
    D(E)= Ker(\Hom(\Omega_X,\mathcal O_X)\oplus \Hom(E_A,
    E)\stackrel{-k\oplus j}{\longrightarrow} \Hom(\Omega_X\otimes E, E))
\end{eqnarray}
Using the explicit structure of $E_A$ as in
(\ref{structure-atiyah}), it's easy to see that $D(E)$ is the sheaf
of pairs on an open subset $U$
\begin{eqnarray*}
    (t_U, \phi_U), \ \ \ \ t_U\in \Hom_{\mathcal O_U}(\Omega_U, \mathcal
    O_U),\ \phi_U\in \Hom_k(E|_U, E|_U )
\end{eqnarray*}
such that
\begin{eqnarray*}
    \phi_U(a\, e)= a\phi_U(e)+t_U(da)e, \ \ a\in \mathcal O_U, \
    e\in E|_U
\end{eqnarray*}

$D(E)$ is known as the sheaf of differential operators of order
$\leq 1$ with diagonal symbol.

\begin{lem}\label{dual relation}
    Let $G^\bullet$ be the mapping cone of the Atiyah class
    $$
        \Hom(E, E)[-1]\to \Omega_X
    $$
    then we have quasi-isomorphism
    $$
        (G^\bullet)^v\simeq D(E)
    $$
    where $D(E)$ is considered as an one-term complex of locally
    free sheaf.
\end{lem}
\begin{proof}
    Let $\alpha(E)$ denote the Atiyah class
    $$
    \alpha(E):  \Hom(E, E)[-1]\to \Omega_X
    $$
    then we have
    $$
         \Hom(E, E)[-1]\stackrel{\alpha(E)}{\to} \Omega_X \to
         G^\bullet\to \Hom(E, E)
    $$
    take the dual, we get
$$
    \Hom(E, E)\to (G^\bullet)^v\to \Hom(\Omega_X, \mathcal
    O_X)\stackrel{(\alpha(E))^v}{\to} \Hom(E, E)[1]
$$
On the other hand, $(\alpha(E))^v: \Hom(\Omega_X, \mathcal O_X)\to
\Hom(E, E)[1]$ is given by the exact sequence
$$
    0\to \Hom(E, E)\to D(E)\to \Hom(\Omega_X, \mathcal O_X)\to 0
$$
which fits into the exact triangle
$$
    \Hom(E, E)\to D(E)\to \Hom(\Omega_X, \mathcal
    O_X)\stackrel{(\alpha(E))^v}{\to} \Hom(E, E)[1]
$$
Comparing the two exact triangles, the lemma follows.
\end{proof}

\begin{lem}
Let $A$ be an Artinian local ring with residue field $k$,
$X_A/\spec\ A$ is a flat deformation of $X/\spec k$
\begin{eqnarray*}
    \xymatrix{
        X\ar@{^(->}[r]\ar[d] & X_A\ar[d]\\
        \spec\ k \ar@{^(->}[r] & \spec\ A
    }
\end{eqnarray*}
let $\tilde E^\bullet$ be a finite complex of locally free sheaves
on $X_A$ whose derived restriction to $X$ is quasi-isomorphic to
$E$, then $\tilde E^\bullet$ is quasi-isomorphic to an one-term
complex of locally free sheaf.
\end{lem}
\begin{proof}
 Let $\tilde E^n$ be the last non-zero term of $\tilde E^\bullet$.
 If $n>0$, then
 $$
    \tilde E^{n-1}|_{X} \to \tilde E^n|_{X}
 $$
 is surjective by assumption. Hence $\tilde E^{n-1}\to \tilde E^n$ is also surjective by
 Nakayama Lemma, and the kernel of $\tilde E^{n-1}\to \tilde E^n$ is then locally
 free. So we can assume that $\tilde E^0$ is the last non-zero
 term.\\
\indent Now let $\tilde E^{n}$ be the first non-zero term. If $n<0$,
consider the map
$$
    \tilde E^{n}\to \tilde E^{n+1}
$$
Let $K$ be the kernel, $I$ be the image, and $Q$ be the cokernel. By
assumption, the map
$$
   \tilde E^n|_X= \tilde E^{n}\otimes_A k\to \tilde E^{n+1}|_X= \tilde E^{n+1}\otimes_A k
$$
is injective. It implies that the map
$$
    \tilde E^{n}\otimes_A k\to I\otimes_A k
$$
is isomorphism and
$$
    0\to I\otimes_A k\to \tilde E^{n+1}\otimes_A k\to Q\otimes_A
    k\to 0
$$
is exact. Since $\tilde E^{n+1}$ is locally free and $X_A$ is flat
over $A$, we get
$$
    Tor_1^A(Q, k)=0
$$
We see that $Q$ is flat over A,
 hence $I$ is flat over $A$ also. Therefore the sequence
$$
    0\to K\otimes_A k \to \tilde E^n\otimes_A k \to I\otimes_A k\to
    0
$$
is exact. We see that $K\otimes_A k=0$. By Nakayama Lemma,
$$
    K=0
$$
Therefore the map $\tilde E^n\to \tilde E^{n+1}$ is both injective
as a map of sheaves and injective on fibers. So $\tilde
E^{n+1}/\tilde E^n$ is also locally free and $\tilde E^\bullet$ can
be trimmed. We can keep this operation until we get one-term complex
of locally free sheaf, which is quasi-isomorphic to $\tilde
E^\bullet$.
\end{proof}

\begin{cor}\label{deform-equiv}
    Let $X$ be projective smooth variety and $E$ a vector bundle on $X$. Then the local deformation functor of the pair $(X, E)$
    viewing $E$ as derived objects on $X$ is isomorphic to the local deformation
    functor of the pair $(X, E)$ viewing $E$ as vector bundle on $X$.
\end{cor}

\begin{thm}
    The tangent space for the deformation of the pair $(X, E)$ is
    given by
    $$
       H^1(X, D(E))
    $$
    and obstruction space can be chosen to be
    $$
 H^2(X, D(E))
    $$
\end{thm}
\begin{proof}Let $G^\bullet$ be the mapping cone of the Atiyah class as above. By corollary \ref{main cor} and corollary \ref{deform-equiv}, the tangent space for the deformation
    of the pair $(X, E)$ is given by
    \begin{eqnarray*}
        \Ext^1_X(G^\bullet, \mathcal O_X)=\Ext^1_X(\mathcal O_X,
        (G^\bullet)^v)
    \end{eqnarray*}
    By lemma \ref{dual relation}
    $$
\Ext^1_X(\mathcal O_X,
        (G^\bullet)^v)=\Ext^1_X(\mathcal O_X,
        D(E))={H}^1(X, D(E))
    $$
    similarly, the obstruction space is given by
    $$
        \Ext^2_X(\mathcal O_X, (G^\bullet)^v)={H}^2(X, D(E))
    $$
\end{proof}

\begin{rmk}
    In the case that $E$ is vector bundle on $X$, this theorem is well-known and can
    also be obtained in the standard way by Cech Cohomology (see for example \cite{deformation} for the line bundle case). Theorem
    \ref{Main Thm} actually generalizes the bundle case above to derived
    objects of coherent sheaves over smooth projective variety.
\end{rmk}

Let $Def_{E}, Def_{(X, E)}, Def_{X}$ be the deformation functor of
$E$, the pair $(X, E)$, and $X$ respectively. Then we have maps
$$
    Def_{E}\to Def_{(X, E)}\to Def_{X}
$$
where the first map is the deformation of $E$ fixing $X$, and the
second map is the forgetful map. Taking the cohomology of the exact
sequence
$$
    0\to \Hom(E, E)\to D(E)\to \Hom(\Omega_X, \mathcal O_X)\to
    0
$$
we get long exact sequence
\begin{eqnarray*}
\xymatrix{
    0\ar[r] & \Ext^0(E, E)\ar[r]& {H}^0(X, D(E)) \ar[r] & H^0(X,
    T_X)\\
    \ar[r] & \Ext^1(E, E)\ar[r] & {H}^1(X, D(E)))\ar[r] & H^1(X,
    T_X)\\
    \ar[r] & \Ext^2(E, E)\ar[r] & {H}^2(X, D(E))\ar[r] & H^2(X, T_X)
    }
\end{eqnarray*}
which can be viewed as
\begin{eqnarray*}
\xymatrix{
  0\ar[r] &  Aut_E\ar[r] & Aut_{(X, E)}\ar[r] & Aut_X\\
    \ar[r] & {Def_E}\ar[r] & {Def_{(X, E)}}\ar[r] & {Def_X}\\
    \ar[r] & Ob_{E}\ar[r] & Ob_{(X, E)}\ar[r] & Ob_{X}
   }
\end{eqnarray*}
Therefore we see that the tangent-obstruction theory for deforming
bundle, deforming pairs of bundle and scheme, and deforming scheme
are naturally combined into long exact sequence coming from exact
triangle in the derived category of coherent sheaves on X via the
construction of Atiyah class. We have the same structure if $E$ is a
perfect complex of coherent sheaves on $X$ by using the exact
triangle (\ref{exact sequence}) instead.

%%%%%%%%%%%%%%%%%%%%%%%%%%%%%%
%% 正文部分
%%%%%%%%%%%%%%%%%%%%%%%%%%%%%%
%\mainmatter

  % 附录
  \appendix

%%%%%%%%%%%%%%%%%%%%%%%%%%%%%%
%% 附件部分
%%%%%%%%%%%%%%%%%%%%%%%%%%%%%%
  % 参考文献
  % 使用 BibTeX

\end{document}